\documentclass[preprint, 10pt, english]{elsarticle}
\usepackage{amsthm}
\usepackage{amsmath}
\usepackage{latexsym, amssymb}
\usepackage{txfonts}
\usepackage{mathtools}
\usepackage{color}
\usepackage{babel}
\usepackage[all]{xy}
\usepackage[capitalise]{cleveref}

\newtheorem{thm}{Theorem}[section] 

\newtheorem{cor}[thm]{Corollary}

\newtheorem{lem}[thm]{Lemma}
\newtheorem{prop}[thm]{Proposition}

\newtheorem{ques}[thm]{Question}

\theoremstyle{definition}
\newtheorem{rem}[thm]{Remark}
\newtheorem{exmpl}[thm]{Example}

\newcommand\operA[2]{{\if!#2!\operatorname{#1}\else{\operatorname{#1}_{#2}^{\phantom{I}}}\fi}} 

\newcommand\tensor[1][]{{\otimes_{#1}}}

\newcommand{\Trace}[1][]{\if!#1!\operatorname{Tr}\else{\operatorname{Tr}_{#1}^{\phantom{I}}}\fi} 

\long\def\forget#1\forgotten{{}} %

\def\({\left(}
\def\){\right)}

\newcommand\Qf[1]{{\left<{#1}\right>}}              
\newcommand\Pf[1]{{\left<\left<{#1}\right>\right>}} 
\newcommand\MPf[1]{{\left<\left<{#1}\right]\right]}} 

\newcommand\LAY[3][]{{\begin{array}{c}\mbox{#2} \if#1!{}\else{+}\fi \\ \mbox{#3}\end{array}}}

\makeatletter

\def\ps@pprintTitle{%
 \let\@oddhead\@empty
 \let\@evenhead\@empty
 \def\@oddfoot{}%
 \let\@evenfoot\@oddfoot}

\newcommand{\bigperp}{%
  \mathop{\mathpalette\bigp@rp\relax}%
  \displaylimits
}

\newcommand{\bigp@rp}[2]{%
  \vcenter{
    \m@th\hbox{\scalebox{\ifx#1\displaystyle2.1\else1.5\fi}{$#1\perp$}}
  }%
}
\makeatother

\renewcommand{\geq}{\geqslant}
\renewcommand{\leq}{\leqslant}

\newif\iffurther
\furtherfalse

\journal{Bulletin of the Australian Mathematical Society}

\begin{document}
\begin{frontmatter}

\title{Common Slots of Bilinear and Quadratic Pfister Forms}

\author{Adam Chapman}
\ead{adam1chapman@yahoo.com}
\address{Department of Computer Science, Tel-Hai Academic College, Upper Galilee, 12208 Israel}

\begin{abstract}
We show that over any field $F$ of $\operatorname{char}(F)=2$ and 2-rank $n$, there exist $2^n$ bilinear $n$-fold Pfister forms that have no slot in common. This answers a question of  Becher's in the negative.
We provide an analogous result also for quadratic Pfister forms.
\end{abstract}

\end{frontmatter}

\section{Introduction}

The study of linkage of quadratic or bilinear $n$-fold Pfister forms and its connections to important field invariants, e.g. the $u$-invariant and the cohomological 2-dimension, has been the focus of several interesting papers in the last five decades. 
The first significant result was obtained in \cite{ElmanLam:1973} where it was shown for nonreal fields $F$ with $\operatorname{char}(F) \neq 2$ that if $I^n F$ is linked (i.e. every two anisotropic $n$-fold Pfister forms have an $(n-1)$-fold Pfister form as a common factor) then $I^{n+2} F=0$, and it was concluded that if $F$ is linked (i.e. $I^2 F$ is linked) then $u(F)$ can be either 0,1,2,4, or 8.
The analogous result for $I_q^n F$ when $\operatorname{char}(F)=2$ was given in \cite{ChapmanDolphin:2017} based on preliminary results obtained in \cite{Faivre:thesis}.

There is an intrinsic complication with quadratic forms when $\operatorname{char}(F)=2$: there exist two kinds of quadratic field extensions - separable and inseparable - which means that a maximal subfield shared by two given quaternion division algebras can be either a separable or inseparable extension of the center, and two quadratic $n$-fold Pfister forms can share either a quadratic or bilinear $(n-1)$-fold Pfister form as a common factor. We specify the terms ``separable" and ``inseparable" linkage accordingly. It was shown that inseparable linkage for $I_q^2 F$ implies separable linkage (\cite{Draxl:1975}) but not vice versa (\cite{Lam:2002}). This fact was generalized to $I_q^n F$ for arbitrary $n$ in \cite{Faivre:thesis} and to symbol division $p$-algebras of arbitrary prime degree in \cite{Chapman:2015}.

In \cite{Becher} the linkage property was extended to larger sets of $n$-fold Pfister forms: we say that $I^n F$ is $m$-linked if every $m$ anisotropic bilinear $m$-fold Pfister forms have an $(n-1)$-fold Pfister form as a common factor. It was shown for nonreal fields $F$ with $\operatorname{char}(F) \neq 2$ that if $I^n F$ is 3-linked then $I^{n+1} F=0$, and concluded that if $I^2 F$ is 3-linked then $u(F) \leq 4$.
The analogous results for $I_q^n F$ when $\operatorname{char}(F)=2$ were obtained in \cite{ChapmanDolphinLeep}.

Becher noticed that there exist fields $F$ for which $I^2 F$ is $m$-linked for any finite $m$ (such as number fields) and asked the following natural question:
\begin{ques}[{\cite[Question 5.2]{Becher}}]\label{q1}
Suppose $I^n F \neq 0$ and $I^n F$ is 3-linked. Does it follow that $I^n F$ is $m$-linked for every finite $m\geq 3$?
\end{ques}

In this paper we provide a negative answer to this question when $\operatorname{char}(F)=2$.
We also consider the analogous questions for $I_q^n F$. We say that $I_q^n F$ is separably (inseparably, resp.) $m$-linked if every $m$ anisotropic quadratic $n$-fold Pfister forms over $F$ have a quadratic (bilinear) $(n-1)$-fold Pfister form as a common factor.
The two analogues of Question \ref{q1} for $I_q^n F$ are:
\begin{ques}\label{q2}
Suppose $I_q^n F \neq 0$ and $I_q^n F$ is inseparably 2-linked. Does it follow that $I_q^n F$ is inseparably $m$-linked for every finite $m \geq 2$?
\end{ques}
\begin{ques}\label{q3}
Suppose $I_q^n \neq 0$ and $I_q^n F$ is separably 3-linked. Does it follow that $I_q^n F$ is separable $m$-linked for every finite $m\geq 3$?
\end{ques}

We answer Question \ref{q2} in the negative, and also Question \ref{q3} for $n \geq 3$.
We conjecture that the answer to \cite[Question 5.2]{Becher} is negative also when $\operatorname{char}(F) \neq 2$.

\section{Preliminaries}

For general reference on symmetric bilinear forms and quadratic forms see \cite{EKM}. 
The group $W_q F=I_q F$ is generated by the forms $\varphi(u,v)=\alpha u^2+u v+\beta v^2$ for $\alpha,\beta \in F$, denoted by $[\alpha,\beta]$. We write $\langle \beta_1,\dots,\beta_n \rangle_b$ for the diagonal bilinear form $$B((v_1,\dots,v_n),(w_1,\dots,w_n))=\sum_{i=1}^n \beta_i v_i w_i$$ and $\langle \beta_1,\dots,\beta_n \rangle$ for the diagonal quadratic form $\varphi(v_1,\dots,v_n)=\sum_{i=1}^n \beta_i v_i^2$.
We denote by $D(\varphi)$ the set of nonzero values $\varphi$ represents, i.e. $\{\varphi(v) :  v \in V, \varphi(v) \neq 0\}$, and by $D(B)$ the set $\{B(v,v) : v \in V, B(v,v) \neq 0\}$. 

The bilinear forms $\Pf{\beta}_b = \Qf{1,\beta}_b$ are called bilinear $1$-fold Pfister forms. These forms generate the basic ideal $IF$ of $WF$. Powers of $IF$ are denoted by $I^nF$. The tensor products $\Pf{\beta_1,\dots,\beta_n}_b = \Pf{\beta_1}_b \tensor \cdots \tensor \Pf{\beta_n}_b$ are called bilinear $n$-fold Pfister forms. 

The quadratic form $[1,\alpha]$ is called a quadratic $1$-fold Pfister form, and denoted by $\MPf{\alpha}$. 
For any quadratic form $\varphi$ and $\beta_1,\dots,\beta_n \in F^\times$, $\Qf{\beta_1,\dots,\beta_n}_b \otimes \varphi = \beta_1 \varphi \perp \dots \perp \beta_n \varphi$.
For any integer $n \geq 2$, we define the quadratic $n$-fold Pfister form $\MPf{ \beta_1,\dots,\beta_{n-1},\alpha}$ as $\Pf{\beta_{1},\dots,\beta_{n-1}}_b \otimes \MPf{\alpha}$.
A quadratic Pfister form is isotropic if and only if it is hyperbolic, and a bilinear Pfister form is isotropic if and only if it is metabolic.
We define $I_q^n F$ to be group generated by the scalar multiples of quadratic $n$-fold Pfister forms.

A quadratic $n$-fold Pfister form $\varphi=\langle \langle \beta_1,\dots,\beta_{n-1},\alpha]]$ over $F$ decomposes as $\varphi=[1,\alpha] \perp \varphi''$.
The quadratic form $\varphi'=\langle 1 \rangle \perp \varphi''$ is independent of the choice of presentation of $\varphi$, and is called the ``pure part" of $\varphi$.
A bilinear form $B=\langle \langle \beta_1,\dots,\beta_n \rangle \rangle_b$ over $F$ decomposes as $B=\langle 1 \rangle_b \perp B'$ for a unique symmetric bilinear form $B'$ called the ``pure part" of $B$.

\section{Bilinear Pfister Forms}\label{Bilinear}

\sloppy Suppose $\operatorname{char}(F)=2$.
We define the 2-rank of $F$ (denoted $\operatorname{rank}_2(F)$) to be $\log_2([F:F^2])$. It is known to be an integer. For any finitely generated field extension $L/F$, $\operatorname{rank}_2(L)=\operatorname{rank}_2(F)+\operatorname{tr.deg}(L/F)$ (see \cite[Lemma 2.7.2]{FriedJarden}).
By \cite[Example 6.5]{EKM}, a given bilinear $n$-fold Pfister form $\langle \langle \beta_1,\dots,\beta_n \rangle \rangle_b$ is anisotropic if and only if $\log_2([F^2(\beta_1,\dots,\beta_n):F^2])=n$. As a result, if $\operatorname{rank}_2(F)=r$ then $I^n F \neq 0$ for all $n \leq r$ and $I^n F=0$ for all $n > r$.

\begin{thm}\label{sharing}
Let $F$ be a field of $\operatorname{char}(F)=2$ and $\operatorname{rank}_2(F)=n$ for some integer $n \geq 2$.
Then for any $m \in \{1,\dots,n-1\}$, every collection of $2^{n-m+1}-1$ anisotropic bilinear $n$-fold Pfister forms have a bilinear $m$-fold Pfister form as a common factor.
\end{thm}

\begin{proof}
Write $N=2^{n-m+1}-1$. Consider $N$ anisotropic bilinear $n$-fold Pfister forms $B_1,\dots,B_N$.
Let $i$ be an integer in $\{0,\dots,m-1\}$.
Suppose there exists a bilinear $i$-fold Pfister form $\rho$ such that $B_\ell=\rho \otimes \pi_\ell$ for some $(n-i)$-fold Pfister forms $\pi_\ell$ for all $\ell \in \{1,\dots,N\}$.
For each $\ell$, $D(\rho \otimes \pi_\ell')$ is an $F^2$-vector subspace of $F$ of dimension $2^n-2^i$.
Since $2^i N \leq 2^n-2^i<2^n$, the spaces $D(\rho \otimes \pi_\ell')$ for $\ell \in \{1,\dots,N\}$ have a nontrivial intersection.
Hence, by \cite[Proposition 6.15]{EKM}, there exists $\beta \in F^\times$ such that $B_\ell=\rho \otimes \langle \langle \beta \rangle \rangle_b \otimes \psi_\ell$ for some bilinear $(n-i-1)$-fold Pfister forms $\psi_\ell$ for all $\ell \in \{1,\dots,N\}$. The statement then follows by induction.
\end{proof}

\begin{cor}
Let $F$ be a field of $\operatorname{char}(F)=2$ with $I^n F \neq 0$ for some $n \geq 2$. Then $I^n F$ is 3-linked if and only if $\operatorname{rank}_2(F)=n$.
\end{cor}

\begin{proof}
Suppose every three anisotropic bilinear $n$-fold Pfister forms over $F$ have a common $(n-1)$-fold Pfister factor. By \cite[Theorem 5.1]{Becher}, $I^{n+1} F=0$, and therefore $\operatorname{rank}_2(F) \leq 2^n$. Since $k_n(F) \neq 0$, $\operatorname{rank}_2(F)=n$.
The opposite direction is Theorem \ref{sharing} with $m=n-1$.
\end{proof}

If we plug in $m=n-1$ in Theorem \ref{sharing}, then it says that when $\operatorname{rank}_2(F)=n \geq 2$, every $2^n-1$ anisotropic bilinear $n$-fold Pfister forms have a common bilinear 1-fold Pfister factor, i.e. a common slot.
The following theorem shows that this bound is sharp by providing $2^n$ bilinear $n$-fold Pfister forms that do not have a common slot.

\begin{thm}\label{MainB}
Let $F$ be a field of $\operatorname{char}(F)=2$ with $\operatorname{rank}_2(F)=n$ for some $n \geq 2$.
Then there exist $2^n$ anisotropic bilinear $n$-fold Pfister forms with no common slot.
\end{thm}

\begin{proof}
Let $\alpha_1,\dots,\alpha_n$ be a 2-basis of $F$ (i.e. $F=F^2(\alpha_1,\dots,\alpha_n)$). Write $I= \{0,1\}^{\times n}$, ${\bf 0}=(0,\dots,0)$, ${\bf d}=(d_1,\dots,d_n)$ for an arbitrary element in $I$, and ${\bf \alpha^d}$ for $\prod_{i=1}^n \alpha_i^{d_i}$. For every ${\bf d} \in I \setminus \{{\bf 0}\}$, let $B_{\bf d}$ be $\langle \langle \alpha_1,\dots,\widehat{\alpha_\ell},\dots,\alpha_n \rangle \rangle_b \otimes \langle \langle 1+{\bf \alpha^d}\rangle\rangle_b$ where $\ell$ is the minimal integer in $\{1,\dots,n\}$ for which $d_\ell \neq 0$.
For every ${\bf e} \in I \setminus \{{\bf 0}\}$ with $e_\ell=0$, both ${\bf \alpha^e}$ and ${\bf \alpha^e}(1+{\bf \alpha^d})={\bf \alpha^e+\alpha^{e+d}}$ are in $D(B_{\bf d}')$, and so also ${\bf \alpha^{e+d}} \in D(B_{\bf d}')$.
Therefore, the elements $\{{\bf \alpha^e}: {\bf e} \in I \setminus \{{\bf 0,d}\}\} \cup \{1+{\bf \alpha^d}\}$ are all in $D(B_{\bf d}')$, and since they are linearly independent over $F^2$ and $D(B_{\bf d}')$ is of dimension $2^n-1$ over $F^2$, they form a basis of $D(B_{\bf d}')$ over $F^2$.

Let $B_{\bf 0}$ be $\langle \langle \alpha_1,\dots,\alpha_n \rangle \rangle_b$, and so $D(B_{\bf 0}')$ is spanned over $F^2$ by $\{\alpha^{\bf e} : {\bf e} \in I \setminus \{{\bf 0}\}\}$. Since $D(B_d)$ for all $d \in I$ are of dimension $2^n$ over $F^2$, they are anisotropic by \cite[Example 6.5]{EKM}. By elementary linear algebra, for any given ${\bf d} \in I \setminus \{ {\bf 0} \}$, the intersection $D(B_{\bf 0}') \bigcap D(B_{\bf d}')$ is spanned by $\{\alpha^{\bf e} : {\bf e} \in I \setminus \{{\bf 0,d}\}\}$, and so the intersection $\bigcap_{{\bf d} \in I} D(B_{\bf d}')$ is trivial.
This means the pure parts of the bilinear $n$-fold Pfister forms $\{ B_{\bf d} : {\bf d} \in I\}$ do not represent a common element, and so they have no slot in common.
\end{proof}

This means the answer to Question \ref{q1} is always negative.
Fields of 2-rank $n$ are easily provided: take any perfect field $F_0$ of $\operatorname{char}(F)=2$, and let $F$ be either the function field $F_0(\alpha_1,\dots,\alpha_n)$ in $n$ algebraically independent variables over $F_0$, or the field of iterated Laurent series $F_0((\alpha_1))\dots((\alpha_n))$ in $n$ variables over $F_0$.

The situation in quadratic forms is more complicated, as we shall see in the next section.
It is a good opportunity to point out another surprising difference between quadratic forms and symmetric bilinear forms in characteristic 2:

\begin{prop}
Suppose two given anisotropic bilinear $n$-fold Pfister forms $B_1$ and $B_2$ over $F$ with $\operatorname{char}(F)=2$ satisfy the following: for every $\alpha \in F^\times$, $\langle \langle \alpha \rangle \rangle$ is a factor of $B_1$ if and only if it is a factor of $B_2$. Then $B_1 \simeq B_2$.
\end{prop}

\begin{proof}
For every $\alpha \in F^\times$, $\langle \langle \alpha \rangle \rangle$ is a factor of $B_1$ if and only if $\alpha$ is represented by $B_1'$. 
If for every $\alpha$, $\alpha$ is represented by $B_1'$ if and only if it is represented by $B_2'$, it means that $D(B_1')=D(B_2')$, i.e. $D(B_1')$ and $D(B_2')$ are the same $(2^n-1)$-dimensional $F^2$-subspace $V$ of $F$. Let $\rho$ be a common $i$-fold factor of $B_1$ and $B_2$. Write $B_1=\rho \otimes \psi_1$ and $B_2=\rho \otimes \psi_2$. The spaces $D(\rho \otimes \psi_1')$ and $D(\rho \otimes \psi_2')$ are $(2^n-2^i)$-dimensional $F^2$-subspaces of $V$. If $i \leq n-1$ then they have a nonzero intersection, because $[V:F^2]=2^n-1$. Let $\beta$ be a nonzero element in the intersection. By \cite[Proposition 6.15]{EKM}, $\rho \otimes \langle \langle \beta \rangle \rangle$ is a common factor of $B_1$ and $B_2$. This works for every $i \in \{1,\dots,n-1\}$. Therefore, we obtain by induction that $B_1 \simeq B_2$.
\end{proof}

This is not true for quadratic $n$-fold Pfister forms, which can  share all 1-fold factors (either bilinear or quadratic, or both) without being isomorphic (see \cite{ChapmanDolphinLaghribi:2016} for reference).

\begin{rem}
In this section we focused on anisotropic bilinear $n$-fold Pfister forms.
By \cite[Page 909]{ArasonBaeza:2007} an isotropic bilinear $n$-fold Pfister form $B$ decomposes as $B=\langle \langle \underbrace{1,\dots,1}_{k \ \text{times}} \rangle \rangle_b \otimes B_1$ where $B_1$ is an anisotropic bilinear $(n-k)$-fold Pfister form, and $D(B)=D(B_1)$ for some unique integer $k$.
However, $B_1$ is not unique, and there can certainly exist a different anisotropic bilinear Pfister form $B_2$ such that $B=\langle \langle 1 \rangle \rangle_b^k \otimes B_2$ as well.
For example, take an anisotropic $B_1=\langle \langle x \rangle 
\rangle_b$, $B_2=\langle \langle x+1 \rangle \rangle_b$ and $B=\langle \langle 1,x \rangle \rangle_b$ (see also \cite[Proposition A.8]{ArasonBaeza:2007}). The situation is therefore more fluid when it comes to isotropic forms. In addition, anisotropic bilinear $n$-fold Pfister forms represent nonzero classes in $I^n F$ and are mapped to nonzero classes in the Milnor $K$-groups $K_n F/2 K_n F$ while all the isotropic $n$-fold Pfister forms are trivial in $I^n F$ and mapped to zero by the isomorphism $I^n F/I^{n+1} F \cong K_n F/2 K_n F$ from \cite{Kato:1982}, which gives anisotropic forms greater significance in the algebraic theory of bilinear forms, $K$-theory and in general.
\end{rem}

\section{Quadratic Pfister Forms}\label{Quadratic}

In this section we provide a negative answer to Question \ref{q2}, and to Question \ref{q3} in all cases but $n=2$. The technique is to study the common quadratic inseparable splitting fields of quadratic $n$-fold Pfister forms.
Given an anisotropic quadratic $n$-fold Pfister form $\varphi$ over $F$ and an inseparable quadratic field $K=F[\sqrt{\gamma}]$, $\varphi_K$ is isotropic if and only if the bilinear 1-fold Pfister form $\langle \langle \gamma \rangle \rangle_b$ is a factor of $\varphi$. Given a quadratic form $\varphi : V \rightarrow F$, a subform $\psi$ of $\varphi$ is the restriction of $\varphi$ to some subspace $W$ of $V$.
\begin{lem}\label{Common}
If $F[\sqrt{\gamma}]$ is a splitting field of an anisotropic quadratic $n$-fold Pfister form $\varphi$ over $F$, then $\langle 1,\gamma \rangle$ is a subform of $\varphi'$.
\end{lem}

\begin{proof}
Follows from \cite[Proposition 3.2]{ChapmanDolphinLaghribi:2016}.
\end{proof}

We focus on valued fields with a sufficiently large value group. For general reference on valuation theory see \cite{TignolWadsworth:2015}. 

\begin{lem}[{\cite[Lemma 10.1]{ChapmanGilatVishne:2017}}]\label{CGV}
Let $n\geq 2$ and $F$ be a field of $\operatorname{char}(F)=2$ with a valuation $\mathfrak{v}$ onto the totally ordered group $\Gamma$.
Write $\overline{\mathfrak{v}}$ for the function mapping each $q \in F^\times$ to the class of $q$ in $\Gamma/2\Gamma$. Let $\alpha_1,\dots,\alpha_n$ be elements in $F^\times$ of negative values whose images under $\overline{\mathfrak{v}}$ are linearly independent over $\mathbb{F}_2$, and consider the quadratic $n$-fold Pfister form $\varphi=\langle \langle \alpha_1,\dots,\alpha_n ]]$ with underlying vector space $V$ with basis $\{v_{\bf d} : {\bf d} \in I\}$, where $I$, ${\bf d}$ and ${\bf\alpha^d}$ be the same as in Theorem \ref{MainB}.
Then 
\begin{itemize}
\item[(a)] For every $v=\sum_{{\bf d} \in I} c_{\bf d} v_{\bf d} \in V$, $\mathfrak{v}(\varphi(v))=\mathfrak{v}(\varphi(c_{\bf d} v_{\bf d}))$ for some specific ${\bf d} \in I$ with $c_{\bf d} \ \neq 0$.
\item[(b)] $\varphi$ is anisotropic.
\item[(c)] $\overline{\mathfrak{v}}(D(\varphi))$ is the $\mathbb{F}_2$-subspace of $\Gamma/2\Gamma$ spanned by $\{\overline{\mathfrak{v}}(\alpha_i) : i \in \{1,\dots,n\}\}$.
\end{itemize}
\end{lem}

\begin{cor}[{\cite[Corollary 10.2]{ChapmanGilatVishne:2017}}]\label{Contr}
For any two dimensional subspace $U$ of $V$, there exists an element in $U$ whose image under $\overline{\mathfrak{v}}$ is nonzero.
\end{cor}

\begin{cor}\label{Miss}
If $\overline{\mathfrak{v}}(D(\varphi))=\Gamma/2\Gamma$, then the form $\varphi$ in Lemma \ref{CGV} satisfies $\{\overline{\mathfrak{v}}(q) : q \in D(\varphi')\}=\Gamma/2 \Gamma \setminus \{\overline{\mathfrak{v}}(\alpha_n)\}$.
\end{cor}

\begin{proof}
Follows immediately from the fact that $\varphi'$ is the restriction of $\varphi$ to the subspace of $V$ spanned by $\{v_{\bf d} : {\bf d} \in I \setminus \{(0,\dots,0,1)\}\}$ and from the linear independence of the images of $\alpha_1,\dots,\alpha_n$ under $\overline{\mathfrak{v}}$ over $\mathbb{F}_2$.
\end{proof}

\begin{thm}\label{Main}
\sloppy Let $n$ be an integer $\geq 2$ and $F$ be a field of $\operatorname{char}(F)=2$ with a discrete rank $n$ valuation.
Then there exist $2^n-1$ quadratic $n$-fold Pfister forms with no common quadratic inseparable splitting field.
\end{thm}

\begin{proof}
Write $\mathfrak{v}$ for the valuation and $\Gamma (\cong \mathbb{Z}^{\times n})$ for the group.
Write $\overline{\mathfrak{v}}$ for the function mapping each $q \in F^\times$ to the class of $q$ in $\Gamma/2 \Gamma$. Let $\alpha_1,\dots,\alpha_n$ be elements in $F^\times$ of negative values whose images under $\overline{\mathfrak{v}}$ are linearly independent over $\mathbb{F}_2$. 
Let $I$, ${\bf 0}$, ${\bf d}$ and ${\bf\alpha^d}$ be the same as in Theorem \ref{MainB}.
For every ${\bf d} \in I \setminus \{{\bf 0}\}$, let $\varphi_{\bf d}$ be $\langle \langle \alpha_1,\dots,\widehat{\alpha_\ell},\dots,\alpha_n \rangle \rangle \otimes \langle \langle {\bf \alpha^d}]]$ where $\ell$ is the minimal integer in $\{1,\dots,n\}$ for which $d_\ell \neq 0$. We will show that the forms $\left\{\varphi_{\bf d} : {\bf d} \in I \setminus \{{\bf 0}\}\right\}$ do not have a common inseparable quadratic splitting field.

By Corollary \ref{Miss}, for each ${\bf d} \in I \setminus \{{\bf 0}\}$, $\overline{\mathfrak{v}}(D(\varphi_{\bf d}'))=\Gamma/2 \Gamma \setminus \overline{\mathfrak{v}}({\bf \alpha^d})$.
Therefore 
$$\bigcap_{{\bf d} \in I \setminus \{{\bf 0}\}} \overline{\mathfrak{v}}(D(\varphi_{\bf d}'))=\{\overline{{\bf 0}}\}.$$
However, by Lemma \ref{Common}, if the forms $\varphi_{\bf d}$ have a common inseparable quadratic splitting field then the forms $\varphi'_{{\bf d}}$ have a common 2-dimensional subform. All the elements $q$ represented by this 2-dimensional subform must satisfy $\overline{\mathfrak{v}}(q)=\overline{\bf 0}$, which contradicts Corollary \ref{Contr}.
\end{proof}

Note that the forms appearing in the statement of Theorem \ref{Main} do not have a bilinear $(n-1)$-fold Pfister form as a common factor, because they do not even share one inseparable quadratic splitting field. When $n \geq 3$ these forms do not have a quadratic $(n-1)$-fold Pfister form as a common factor for the same reason.

For the construction of counterexamples for Question \ref{q2} we need a necessary condition for $I_q^n F$ to be separably 3-linked.

\begin{lem}\label{rightfactor}
Let $\varphi=\langle \langle a_1,\dots,a_n ]]$ be an $n$-fold quadratic Pfister over a field $F$ with $\operatorname{char}(F)=2$.
Write $\varphi=[a_n,1] \perp \varphi''$
and consider $d \in D(\varphi)$ such that
$d=\varphi(w,x,u_1,\dots,u_{2^n-2})=a_n w^2+w x+x^2+\varphi''(u_1,\dots,u_{2^n-2})$ for some $w,x,u_1,\dots,u_{2^n-2} \in F$ with $w \neq 0$.
Then there exist $b_1,\dots,b_{n-1} \in F^\times$ such that
$\varphi=\langle \langle b_1,\dots,b_{n-1},\frac{d}{w^2}]]$.
\end{lem}

\begin{proof}
Let $v_1$ be the vector $(1,\frac{x}{w},\frac{u_1}{w},\dots,\frac{u_{2^n-2}}{w})$, and $v_2$ be the vector $(0,1,0,\dots,0)$.
Then the subform $\varphi|_{F v_1+F v_2}$ is isometric to $[a_n+\frac{x}{w}+\frac{x^2}{w^2}+\varphi''(\frac{u_1}{w},\dots,\frac{u_{2^n-2}}{w}),1]$. 
By \cite[Chapter 4, Lemma 4.1]{Baeza:1978}, there exist $b_1,\dots,b_{n-1} \in F$ such that 
$$\varphi=\langle \langle b_1,\dots,b_{n-1},a_n+\frac{x}{w}+\frac{x^2}{w^2}+\varphi''(\frac{u_1}{w},\dots,\frac{u_{2^n-2}}{w})]].$$
\end{proof}

Recall the $u(F)$ is the maximal dimension of an anisotropic nonsingular quadratic form over $F$ (see \cite[Page 163]{EKM}).

\begin{prop}[{cf. \cite[Corollary 5.4]{Becher}}]\label{Triple}
Let $n$ be an integer $\geq 3$ and $F$ be a field of $\operatorname{char}(F)=2$ such that the function field $K=F(t)$ in one variable over $F$ has $u(K) \leq 2^{n+1}$. Then $I_q^n F$ is separably 3-linked.
\end{prop}

\begin{proof}
Let $\varphi_1,\varphi_2$ and $\varphi_3$ be three anisotropic quadratic $n$-fold Pfister forms over $F$.
Write $\varphi_i=[1,\alpha_i] \perp \varphi_i''$ for $i \in \{1,2,3\}$.
The system of two quadratic equations
$$\alpha_1 w^2+\varphi_1''(v_1)=\alpha_2 w^2+w x_2+x_2^2+\varphi_2''(v_2)$$
$$\alpha_1 w^2+\varphi_1''(v_1)=\alpha_3 w^2+w x_3+x_3^2+\varphi_3''(v_3)$$
has a solution over $F$ if and only if the quadratic form $$\psi : K \times K \times K \times K^{\times (2^n-2)} \times K^{\times (2^n-2)} \times K^{\times (2^n-2)} \rightarrow K$$ mapping $(w,x_2,x_3,v_1,v_2,v_3)$ to
$$\alpha_1 w^2+\varphi_1''(v_1)+\alpha_2 w^2+w x_2+x_2^2+\varphi_2''(v_2)+t(\alpha_1 w^2+\varphi_1''(v_1)+\alpha_3 w^2+w x_3+x_3^2+\varphi_3''(v_3))$$
is isotropic by \cite[Theorem 17.14]{EKM}.
The form $\psi$ is of dimension $3 \cdot (2^n-1)$ which is greater than $2^{n+1}$.
Therefore $\psi$ is isotropic ($u(K) \leq 2^{n+1}$), and so the system above has a solution over $F$. 
If in this solution $w \neq 0$ then by Lemma \ref{rightfactor} the forms $\varphi_1, \varphi_2$ and $\varphi_3$ have a common right slot, i.e. $\varphi_i=\rho_i \otimes \langle \langle \alpha]]$ for $i \in \{1,2,3\}$ for some $\alpha \in F$ and bilinear $(n-1)$-fold Pfister forms $\rho_1,\rho_2$ and $\rho_3$. If the solution has $w=0$ then by \cite[Lemma 3.5]{AravireBaeza:1992}, $\varphi_1$, $\varphi_2$ and $\varphi_3$ have a common bilinear 1-fold Pfister form as a common factor. By \cite[Corollary 6.2]{ChapmanMcKinnie}, since $n \geq 3$, the forms $\varphi_1, \varphi_2$ and $\varphi_3$ also have a common right slot, so they have a common right slot regardless of $w$.

Write $\varphi_i=B_i \otimes \rho$ for $i \in \{1,2,3\}$ some bilinear $(n-k)$-fold Pfister forms $B_1,B_2,B_3$ and some quadratic $k$-fold Pfister form $\rho$ where $k$ is an integer in $\{1,\dots,n-2\}$. The system of two equations
$$(B_1' \otimes \rho)(v_1)=(B_2' \otimes \rho)(v_2)$$
$$(B_1' \otimes \rho)(v_1)=(B_3' \otimes \rho)(v_3)$$
has a solution over $F$ if and only if the quadratic form
$$\theta : K^{\times (2^n-2^k)} \times K^{\times (2^n-2^k)} \times K^{\times (2^n-2^k)}$$
mapping $(v_1,v_2,v_3)$ to 
$$(B_1' \otimes \rho)(v_1)+(B_2' \otimes \rho)(v_2)+t((B_1' \otimes \rho)(v_1)=(B_3' \otimes \rho)(v_3))$$
is isotropic by \cite[Theorem 17.14]{EKM}.
The dimension of $\theta$ is $3 \cdot (2^n-2^k)$ which is greater than $2^{n+1}$ because $k \leq n-2$. Therefore by \cite[Lemma 3.5]{AravireBaeza:1992} there exists $\gamma \in F^\times$ such that $\langle \langle \gamma \rangle \rangle \otimes \rho$ is a common factor of $\varphi_1,\varphi_2$ and $\varphi_3$.
The statement then follows by induction.
\end{proof}

We are now ready to give negative answers to Questions \ref{q2} and \ref{q3}:
\begin{exmpl}
Let $F_0$ be an algebraically closed field of $\operatorname{char}(F_0)=2$, such as the separable closure of $\mathbb{F}_2$, and let $F$ be either the function field $F_0(\alpha_1,\dots,\alpha_n)$ in $n$ algebraically independent variables, or the field $F_0((\alpha_1^{-1}))\dots((\alpha_n^{-1}))$ of iterated Laurent series in $n$ variables over $F_0$.
In these cases the maximal dimension of an anisotropic form in $I_q^n F$ is $2^n$, so $I_q^n F$ is inseparably 2-linked.
However, $F$ has a discrete rank $n$ valuation, and therefore there exist $(2^n-1)$ quadratic $n$-fold Pfister forms without a common quadratic inseparable splitting field, providing a negative answer to Question \ref{q2}.
Moreover, these fields are $C_n$ fields (\cite[Section 97]{EKM}), and therefore $u(F(t))=2^{n+1}$.
By Proposition \ref{Triple}, when $n \geq 3$, $I_q^n F$ is separably 3-linked. However, $I_q^n F$ is not separably $(2^n-1)$-linked for the reason mentioned above, giving a negative answer to Question \ref{q3} (when $n \geq 3$).
\end{exmpl}

Our ability to answer Question \ref{q3} when $n \geq 3$ relies heavily on the fact that when $n \geq 3$, quadratic $n$-fold Pfister forms with a common quadratic $(n-1)$-fold Pfister factor must have a common inseparable quadratic splitting field.
This is certainly not true for $n=2$, and we leave Question \ref{q3} in this case open.
The existence of inseparable quadratic field extensions is special to the case of $\operatorname{char}(F)=2$, so our techniques do not apply (at least not in an obvious manner) to the more common case of $\operatorname{char}(F) \neq 2$.
\section{Quaternion Algebras}\label{Quaternion}

Given a field $F$ of $\operatorname{char}(F)=2$, a quaternion algebra over $F$ is of the form 
$$(\beta,\alpha]_{2,F}=F \langle x,y : x^2+x=\alpha, y^2=\beta, y x y^{-1}=x+1 \rangle$$
for some $\alpha \in F$ and $\beta \in F^\times$. There is a one-to-one correspondence between quaternion algebras $(\beta,\alpha]_{2,F}$ and their norm forms $\langle \langle \beta,\alpha ]]$ which are quadratic 2-fold Pfister forms (see \cite[Section 12]{EKM} and \cite[Section 6]{ChapmanDolphinLaghribi:2016}). In particular, the splitting  fields of the quaternion algebra and its norm form are the same. 

We therefore obtain the following:

\begin{thm}\label{quat}
Let $F$ be a field of $\operatorname{char}(F)=2$ with a rank 2 valuation $\mathfrak{v}$ and value group $\Gamma (\cong \mathbb{Z} \times \mathbb{Z})$.
Write $\overline{\mathfrak{v}}$ for the function mapping each $q \in F^\times$ to the class of $q$ in $\Gamma/2 \Gamma$. Let $\alpha,\beta$ be elements in $F^\times$ of negative values whose images under $\overline{\mathfrak{v}}$ are linearly independent over $\mathbb{F}_2$. Let $Q_1=(\beta,\alpha]_{2,F}$, $Q_2=(\alpha,\beta]_{2,F}$ and $Q_3=(\beta,\alpha \beta]_{2,F}$. Then $Q_1$, $Q_2$ and $Q_3$ do not have a common inseparable quadratic splitting field.
\end{thm}

Fields $F$ with $u(F)=4$ are the fields over which every two quaternion algebras share an inseparable quadratic splitting field (\cite[Theorem 3.1]{Baeza:1982}). If every three quaternion algebras over $F$ share an inseparable quadratic splitting field, it does not affect the $u(F)$. Nevertheless, there exist fields that do not have this property while still having $u(F)=4$:

\begin{exmpl}
Let $F_0$ be an algebraically closed field of $\operatorname{char}(F_0)=2$. Let $F$ be either the function field $F_0(\alpha,\beta)$ in $2$ algebraically independent variables over $F_0$, or the field of iterated Laurent series $F_0((\alpha^{-1}))((\beta^{-1}))$ in $2$ variables over $F_0$. Then every pair of quaternion algebras over $F$ share a quadratic inseparable splitting field, but not every triple.
\end{exmpl}

\begin{proof}
The field $F$ in both cases is a $C_2$ field (see \cite[Section 97]{EKM}) with nontrivial quaternion algebras, and so $u(F)=4$.
Therefore every two quaternion algebras over $F$ share a quadratic inseparable splitting field. However, by Theorem \ref{quat} there exist three quaternion algebras that do not share a quadratic inseparable splitting field.
\end{proof}

There are still fields over which every collection of quaternion algebras share a quadratic inseparable splitting field, as the following example demonstrates. This means that unlike Question \ref{q1}, the answer to Question \ref{q2} is not always negative.

\begin{exmpl}
Let $F_0$ be a perfect field of $\operatorname{char}(F_0)=2$ with nontrivial $\text{\'Et}_2(F)$ (e.g. any finite field). Let $F$ be either the function field $F_0(\alpha)$ in one variable over $F_0$, or the field of Laurent series $F_0((\alpha))$ over $F_0$. Then any finite number of quaternion algebras over $F$ share a quadratic inseparable splitting field, because $F$ has a unique quadratic inseparable field extension.
\end{exmpl}

\section*{Acknowledgements}

We thank Adrian Wadsworth for his comments on the manuscript.
We also thank the anonymous referee for the very helpful comments and suggestions which improved both the quality of the paper and its readability and coherency.

\section*{Bibliography}
\def\cprime{$'$}

\end{document}